\documentclass[11pt]{article}
\usepackage{amssymb, amsmath, latexsym, amsbsy}
\usepackage{amsfonts}
\usepackage{dsfont}
\usepackage[T1]{fontenc}
\usepackage[all]{xy}
\usepackage{amscd}
\usepackage{graphicx}
\usepackage{latexsym}

\usepackage{pxfonts}
\usepackage{enumerate}
\usepackage{color}
\usepackage[colorlinks]{hyperref}
\usepackage{fancybox}
\usepackage{fancyhdr}
\usepackage{lipsum}

\newtheorem{Definition}{Definition}

\newtheorem{Lemma}{Lemma}
\newtheorem{Corollary}{Corollary}
\newtheorem{Remark}{Remark}
\newtheorem{Proposition}{Proposition}
\newenvironment{proof}[1][Proof]{\textbf{#1:} }{\ \rule{0.5em}{0.5em}}
\date{\empty}

\date{}
\title{Schouten like metrics on five dimensional nilpotents Lie groups.}

\begin{document}
	
	\maketitle
	\begin{center}
		\author{ \textbf{M. L. Foka}$^{1}$, \quad \textbf{M. B. N.  Djiadeu}$^{2}$,\quad \textbf{T.B. Bouetou}$^{3}$,\\
			\small{e-mail: $\textbf{1}.$ lanndrymarius@gmail.com, \quad
				$\textbf{2}.$ michel.djiadeu@facsciences-uy1.cm, \\
				$\textbf{3}.$ tbouetou@gmail.com,  \\
				University of Yaounde 1, Faculty of Science, Department of
				Mathematics,  P.O. Box 812, Yaounde, Republic of Cameroon.}}
	\end{center}

	
	\begin{abstract}
	The prescribed Ricci curvature problem consists in finding a Riemannian metric $g$ to satisfy the equation $Ric(g) = T$,
	for some fixed symmetric $(0,2)$-tensor field $T$ on a differential manifold $M$. 	
	In this paper, we define Schouten like metric as a particular solution of a prescribed Ricci curvature problem and we classify them on five  dimensional
	nilpotent Lie groups by establishing a link with algebraic Schouten solitons. This also permit us to make a classification of five dimensional nilsoliton.
	\end{abstract}
	
	\textbf{keywords}: Lie groups, Lie algebra, Ricci operator, nilsoliton, algebraic Schouten solitons,Schouten like metric.

	
\section{Introduction} In 2023, SIYAO Liu introduced in \cite{si} the notion of Algebraic Schouten solitons based on the general definition of the Schouten tensor in \cite{cal}. He was strongly motivated by \cite{cal,we}.
However the Ricci curvature problem is a problem that arouses great interest in Riemannian geometry. Indeed, it is about the search for Riemannian metrics on a differentiable manifold whose Ricci curvature is a tensor field of type (0,2) prescribed on the manifold. For a review of the literature on this problem, see the references \cite{butt, bp, mil, pu1}. One of the most recent works on this problem is that of FOKA .M .L and al in \cite{foka}.
Using the definition of Schouten solitons introduced in \cite{cal}, we introduce the notion of Schouten likes metrics as the solutions of a prescribed Ricci curvature problem in which we impose a particular form on the prescribed tensor field. In this paper we devote ourselves to the classification of Schouten likes metrics on five-dimensional nilpotent Lie groups using the recent classification of five-dimensional nilpotent Lie groups given in \cite{foka}. We give as a direct correlative of this classification, a classification of the dimension five Nilsolitons whose first classification is given by MOGHADDAM and Hamid Reza Salimi in \cite{mo}.

To carry out our work we organize it as follows. In Section 2 we unfold preliminary tools necessary for the establishment of our results. Section 3 will be dedicated to the study proper of Schouten likes metrics on nilpotent Lie groups of dimension five.
\section{Preliminaries}
In this section, we fix our notation and recall some useful facts that will be used
to establish our results in the next section.
Let $(G,g)$ be a finite dimensional Lie group equipped with a left-invariant riemanniann metric, and $\mathfrak{g}$ it Lie algebra. 

For any $u\in\mathfrak{g}$, we denote by $ad_u, \mathsf{J}_u:\mathfrak{g}\longrightarrow\mathfrak{g}$ the endomorphisms given by $ad_uv=[u,v]$ and,$\mathsf{J}_uv=ad^*_vu$ where $ad^*_v$ is the transpose of $ad_v$. We denote also by $B : \mathfrak{g}\times \mathfrak{g}\longrightarrow\mathfrak{g}$ the Killing form given by
\begin{equation*}
	B(u,v)=tr(ad_u\circ ad_v) \quad \forall u,v\in\mathfrak{g}.
\end{equation*}
The mean curvature vector on $\mathfrak{g}$ is defined by
\begin{equation*}
	\langle H,u\rangle = tr(adu), \forall u \in \mathfrak{g} \quad\big(\langle.,.\rangle=g(e)\big).
\end{equation*}
The Ricci curvature tensor $ric_g$ and the Ricci operator $Ric_g$ of $(G,g)$ are given by (see \cite{buc}),
\begin{equation*}
	ric_g(u,v)=-\frac{1}{2}B(u,v)-\frac{1}{2}tr(ad_u\circ ad^*_v)-\frac{1}{4}tr(\mathsf{J}_u\circ\mathsf{J}_v)-\frac{1}{2}(\langle ad_Hu, v\rangle+\langle ad_Hv, u\rangle)
\end{equation*}
and
\begin{equation}
	ric_g(u, v) = g(Ric_g(u), v) \;\textbf{for}\; u,v \in \mathfrak{g}
\end{equation}
\begin{Remark} If
	\begin{itemize}
		\item[i-] If $(G, g)$ is a Riemannian nilpotent Lie group. Then its Ricci
	curvature tensor is given by
		\begin{equation}\label{riccicur}
		ric_g(u,v)=-\frac{1}{2}tr(ad_u\circ ad^*_v)-\frac{1}{4}tr(\mathsf{J}_u\circ\mathsf{J}_v).
	\end{equation}
	\item [ii-] Since $ric_g$ is symetric, $Ric_g$ is self-adjoint. 
	
	\end{itemize}

\end{Remark}
\begin{Definition}
	A derivation $D\in Der(\mathfrak{g})$ is said to be symmetric with respect to the metric $g$ if it satisfies the following condition:
	\begin{equation}
		g(D(u),v)=g(u,D(v))\; \forall u,v\in\mathfrak{g}.
	\end{equation}
\end{Definition} 
A generalized definition of the Schouten tensor is given in \cite{si} by
\begin{equation}
	S(X, Y) = 	ric_g(X, Y) - s\lambda_0 g(X, Y)
\end{equation}
where s denotes the scalar curvature and $\lambda_0$ is a real number. 
\begin{Definition}
	A left-invariant Riemanniann metric $g$ on a Lie group $G$ of Lie algebra $\mathfrak{g}$ is called a nilsoliton if it satisfies
	\begin{equation}\label{eqnil}
		Ric_g \in \mathbb{R} Id_{\mathfrak{g}}\oplus Der(\mathfrak{g}).
	\end{equation}Where $Ric_g$ is the Ricci operator of $(G,g)$ and $Der(\mathfrak{g})$ the set of derivations of $\mathfrak{g}$.\end{Definition}
\begin{Definition}
	A left-invariant Riemanniann metric $g$ on a Lie group $G$ of Lie algebra $\mathfrak{g}$ is called an algebraic Schouten soliton associated with the connection $\bigtriangledown$ if
	it satisfies
	\begin{equation}\label{eqsch}
		Ric_g -(s\lambda_0 + c)Id\in Der(\mathfrak{g}).
	\end{equation}
	
	Where $\bigtriangledown$ is the Levi-Civita connection of $(G,g)$, $Ric_g$  the Ricci operator of $(G,g)$, $s$ the scalar curvature of $(G,g)$, $c, \lambda_0$ and $c$ are real numbers and $Der(\mathfrak{g})$ the set of derivations of $\mathfrak{g}$.
\end{Definition}
\begin{Proposition} \label{propoo}
	If $D=Ric_g -(s\lambda_0 + c)Id$ is a derivation of $\mathfrak{g}$ then it's symetric with respect to $g$.
\end{Proposition}
\begin{proof}
	Let's suppose that $D=Ric_g -(s\lambda_0 + c)Id$ is a derivation of $\mathfrak{g}$. Let $u,v\in \mathfrak{g}$.
	\begin{eqnarray*}
		g(D(u),v)&=&g(Ric_g(u)-(s\lambda_0 + c)u,v)\\
		&=&g(Ric_g(u),v)-(s\lambda_0 + c)g(u,v)\\
		&=&ric_g(u,v)-(s\lambda_0 + c)g(u,v)\\
		&=&ric_g(v,u)-(s\lambda_0 + c)g(v,u)\\
		&=&g(Ric_g(v),u)-(s\lambda_0 + c)g(v,u)\\
		&=&g(Ric_g(v)-(s\lambda_0 + c)v,u)\quad \text{because $Ric_g$ is self-adjoint}\\
		&=&g(D(v),u)\\
		&=&g(u,D(v))
	\end{eqnarray*}
\end{proof}
\begin{Proposition}
	Nilsolitons and algebraic Schouten solitons are invariant
	under isometry and scaling.\end{Proposition}
\begin{proof}\label{rem1}
	Indeed, Let $g$ and $h$ be two left-invariant riemanniann metrics on $G$ such that there exist $c>0$, $\varphi \in Aut(\mathfrak{g})$ satisfying
	$h=c\varphi. g.$
	Then $Ric_h = c^2\varphi Ric_g \varphi^{-1}$.\\
	As $\varphi Der(\mathfrak{g})\varphi^{-1} = Der(\mathfrak{g})$, it follows that $Ric_g \in \mathbb{R} Id_{\mathfrak{g}}\oplus Der(\mathfrak{g})$ if and only if $Ric_h \in \mathbb{R} Id_{\mathfrak{g}}\oplus Der(\mathfrak{g})$.
	The proof is similar for algebraic Schouten solitons. 
\end{proof}
\begin{Remark}\label{rem}
	Nilsolitons constitute a particular class of algebraic Schouten solitons. This because by taking $\lambda_0=0$	in the definition of algebraic Schouten solitons, we obtain Nilsolitons.
\end{Remark}

\begin{Definition}
	A left-invariant Riemanniann metric $g$ on a Lie group $G$ of Lie algebra $\mathfrak{g}$ is called a Schouten like metric if it satisfies
	\begin{equation}\label{eqnil}
		ric_g(\cdot,\cdot) =(s\lambda_0+c)g(\cdot,\cdot)+g(D(\cdot),\cdot) .
	\end{equation}
	Where $s$ is the scalar curvature, $\lambda_0$ and $c$ are real numbers, and $D$ is symetric derivation of $\mathfrak{g}$ with respect to $g$.
\end{Definition}
\begin{Lemma}\label{lem}
	A left-invariant Riemanniann metric $g$ on a Lie group $G$ of Lie algebra $\mathfrak{g}$ is a	Schouten like metric if and only if it's an algebraic Schouten soliton
\end{Lemma}
\begin{proof}Let $g$ be a left-invariant Riemanniann metric on a Lie group $G$ of Lie algebra $\mathfrak{g}$.\\ 
	If $g$ is a Schouten like metric on $G$ then for all $X,Y\in \mathfrak{g}$, $$ric_g(X,Y) =(s\lambda_0+c)g(X,Y)+g(D(X),Y)\; \lambda_0, c\in \mathbb{R}.$$ From the definition of the Ricci operator, we have
	$$g(Ric_g(X), Y)=(s\lambda_0+c)g(X,Y)+g(D(X),Y). $$ Therefore the fact that $g$ is bilinear gives,
	$$g(Ric_g(X), Y)=g\big((s\lambda_0+c)X+D(X),Y\big).$$ Since $g$ is non-degenerated, $Ric_g(X)-(s\lambda_0+c)X=D(X)$ for all $X\in \mathbb{R}$. So $g$ is an algebraic Schouten soliton on $G$. \\
	Conversely, let us assume that $g$ is an algebraic Schouten soliton on $G$. Then there exist $c\in \mathbb{R}$ and $D\in Der(\mathfrak{g})$ such that 
	\begin{equation}\label{eqsch}
		Ric_g -(s\lambda_0 + c)Id= D, \;\lambda_0\in \mathbb{R} .
	\end{equation} 
	Therefore, for all $X,Y\in \mathfrak{g}$,
	$$g(Ric_g(X), Y)=g\big((s\lambda_0+c)X+D(X),Y\big).$$
	Then, for all $X,Y\in \mathfrak{g}$,
	$$ric_g(X,Y)=g(Ric_g(X), Y)=(s\lambda_0+c)g(X,Y)+g(D(X),Y). $$ 
	Moreover, Proposition \ref{propoo} ensures the symmetry of $D$. Consequently, $g$ qualifies as a Schouten-like metric on $G$.
\end{proof}

Marius Landry Foka et al., \cite{foka} characterised the set of all inner products on each five-dimensional nilpotent Lie algebra using the Milnor-type theorem technique. For a detailed exploration of this technique, refer to \cite{hash, taka, koda}. The characterization is encapsulated in the following result.
\begin{Proposition}\cite{foka}\label{proposition 1}
	
	Let $\mathfrak{g}$ be a five-dimensional nilpotent Lie algebra. For every inner product $\langle,\rangle$ on $\mathfrak{g}$, there exist $\eta> 0$,
	and an orthonormal basis $\mathcal{B}=\{v_1, v_2, v_3,v_4,v_5\}$ with respect to $\eta\langle\cdot,\cdot\rangle$, such that all non zero commutators are given by Table \ref{table2}:
	\begin{table}[!h]
		\centering
		\caption{}
		\label{table2}
		{\begin{tabular}{p{1.75cm}p{8cm}p{2.5cm}}
				\hline
				Lie algebra &\hspace{1.5cm}non zero commutation relation&  conditions   \\ 
				\hline$5A_1$  & \hspace{3cm} none & \hspace{1cm} none  \\	
				\hline
				\vspace{0.35cm}$A_{5,4}$ &	
				\begin{itemize}	
					\item[] $[v_1,v_3]=\alpha v_5$,\; $[v_1,v_4]=\beta v_5$
					\item[] $[v_2,v_3]=\gamma v_5$
				\end{itemize}&   \vspace{0.35cm}$\beta,\gamma>0$,\; $\alpha\in \mathbb{R}$ \\
				\hline
				\vspace{0.25cm}$ A_{3,1}\oplus 2A_{1}$& 	\begin{itemize}	
					\item[] $[v_1,v_2]=\alpha v_5$\end{itemize}& \vspace{0.25cm}$\alpha >0$\\
				\hline 
				\vspace{0.25cm}$A_{4,1}\oplus A_{1}$ \textbf{first case}&\begin{itemize}
					\item[]
					$[v_1,v_2]=\alpha v_3+\gamma v_5$,\; $[v_1,v_3]=\beta v_5$
				\end{itemize}&\vspace{0.25cm}$\alpha,\beta>0$,\; $\gamma\in \mathbb{R}$\\
				\hline 
				\vspace{0.25cm}$A_{4,1}\oplus A_{1}$ \textbf{second case}&\begin{itemize}
					\item[]$[v_1,v_2]=\alpha v_3+\gamma v_4$,\; $[v_1,v_3]=\beta v_5$\end{itemize}
				&\vspace{0.25cm}$\alpha,\beta>0$,\; $\gamma\in \mathbb{R}$\\
				\hline
				\vspace{0.5cm}$A_{5,6}$&\begin{itemize}
					\item[]$[v_1,v_2]=\alpha v_3+\beta v_4$,\; $[v_1,v_3]=\gamma v_4+ \delta v_5$
					\item[] $[v_1,v_4]= \varepsilon v_5$,\;$[v_2,v_3]= \sigma v_5$
				\end{itemize}& \vspace{0.25cm} 
				$\alpha<0$, $\gamma,\varepsilon,\sigma>0$, $\beta, \delta\in \mathbb{R}$\\
				\hline
				\vspace{0.25cm}$A_{5,5}$ &\begin{itemize}
					\item[] $[v_1,v_2]=\alpha v_4+\beta v_5$,\; $[v_1,v_3]=\gamma v_5$
					\item[]  $[v_2,v_3]=\delta v_5$,\;$[v_2,v_4]=\varepsilon v_5$
				\end{itemize}&\vspace{0.25cm}$\alpha,\gamma, \varepsilon>0$,\; $\beta,\delta\in \mathbb{R}$ \\
				\hline
				\vspace{0.5cm}$A_{5,3}$ &\begin{itemize}
					\item[] $[v_1,v_2]=\alpha v_3+\beta v_4$,\; $[v_1,v_3]=\gamma v_4+ \delta v_5$\item[]$[v_2,v_3]= \varepsilon v_5$
				\end{itemize} &\vspace{0.5cm}$\alpha,\gamma, \varepsilon>0$,\; $\beta,\delta\in \mathbb{R}$ \\
				\hline
				\vspace{0.25cm}$ A_{5,1}$& \begin{itemize}
					\item[]$[v_1,v_2]=\alpha v_4+\beta v_5$,\; $[v_1,v_3]=\gamma v_5$\end{itemize} &\vspace{0.25cm}$\alpha,\gamma >0$,\; $\beta\in \mathbb{R}$\\
				\hline
				\vspace{0.25cm}$ A_{5,2}$& \begin{itemize} \item[]$[v_1,v_2]=\alpha v_3+\beta v_4$,\; $[v_1,v_3]=\gamma v_4$
					\item[] $[v_1,v_4]= \delta v_5$\end{itemize}  &\vspace{0.25cm}$\alpha,\gamma, \delta>0$,\; $\beta,\in \mathbb{R}$ \\
				\hline
		\end{tabular}}
	\end{table}
\end{Proposition}	

	\section{Schouten like metrics on Five dimensional nilpotent Lie  groups }
	Throughout this section, by  $\mathfrak{g}$ we shall denote any of the five-dimensional nilpotent Lie algebra cited in Table \ref{table2} and $G$ the five-dimensional simply connected nilpotent Lie group having Lie algebra $\mathfrak{g}$. \\
	
	From Lemma \ref{lem}, one can observe that the study of Schouten like metrics can be done through Schouten solitons.
	For this purpose let introduce the following endomorphism $D$ of $\mathfrak{g}$ with $g$ a left-invariant Riemanniann metric on $G$ \begin{eqnarray*}
		D:\mathfrak{g}&\longrightarrow&\mathfrak{g}\\
		X&\longmapsto&\big(Ric_g-(\lambda_0 s+c) Id_{\mathfrak{g}}\big)X.
	\end{eqnarray*} Where $\lambda_0$ and $c$ are real constants and $Ric_g$ the Ricci operator of $(G,g)$. The metric $g$ will be a Schouten like metric if and only if $D$ satisfy the following equation:
	\begin{equation}\label{deriv}
		D[X,Y]= [DX,Y]+[X,DY]\; \textbf{for} X,Y\in \mathfrak{g}.
	\end{equation} 
	The corollaries of the following proposition come directly from the Remark \ref{rem}. 
	\begin{Proposition}
		An inner product $\langle,\rangle$ on $\mathfrak{g}=A_{5,4}$ is a Schouten like metric if and only if $\alpha=0$ and $\beta=\gamma.$ 
	\end{Proposition}
	\begin{proof}
		Let $\langle,\rangle$ be an inner product on $A_{5,4}$. From \cite{foka}, \begin{equation}
			Ric_{\langle,\rangle}=-\frac{1}{2}{\renewcommand{\arraystretch}{0.4}\begin{pmatrix}
					\alpha^2+\beta^2&\alpha\gamma&0&0&0\\
					\alpha\gamma&\gamma^2&0&0&0\\
					0&0&\alpha^2+\gamma^2&\alpha\beta&0\\
					0&0&\alpha\beta&\beta^2&0\\
					0&0&0&0&-\alpha^2-\beta^2-\gamma^2
			\end{pmatrix}}.
		\end{equation}
		Therefore, we can write $D$ as
		\begin{equation}
			\left\{{\renewcommand{\arraystretch}{1}\begin{array}{%
						l @{\qquad} r @{~~} l}
					Dv_1=-\frac{1}{2}[(1-\lambda_0)\alpha^2+(1-\lambda_0)\beta^2-\lambda_0\gamma^2+2c ]v_1 -\frac{1}{2}\alpha\gamma v_2&\\
					Dv_2=-\frac{1}{2}\alpha\gamma v_1-\frac{1}{2}[-\lambda_0\alpha^2-\lambda_0\beta^2+(1-\lambda_0)\gamma^2+2c]v_2 & \\
					Dv_3=-\frac{1}{2}[(1-\lambda_0)\alpha^2-\lambda_0\beta^2+(1-\lambda_0)\gamma^2+2c] v_3 -\frac{1}{2}\alpha\beta v_4 &\\
					Dv_4=-\frac{1}{2}\alpha\beta v_3-\frac{1}{2}[-\lambda_0\alpha^2+(1-\lambda_0)\beta^2-\lambda_0\gamma^2+2c ]v_4  &\\
					Dv_5=\frac{1}{2}[(1+\lambda_0)\alpha^2+(1+\lambda_0)\beta^2+(1+\lambda_0)\gamma^2-2c]v_5. 
			\end{array}}\right.
		\end{equation}	
		Hence, by \eqref{deriv}, $\langle,\rangle$ is an algebraic Schouten soliton if and only if the following polynomial system of equations is satisfied
		
		\begin{eqnarray}
			\alpha[(3-\lambda_0)\alpha^2+(3-\lambda_0)\beta^2+(3-\lambda_0)\gamma^2+2c ]&=&0\nonumber\\
			\beta[(3-\lambda_0)\alpha^2+(3-\lambda_0)\beta^2+(1-\lambda_0)\gamma^2+2c]&=&0\nonumber\\
			\gamma[(3-\lambda_0)\alpha^2+(1-\lambda_0)\beta^2+(3-\lambda_0)\gamma^2+2c]&=&0	\label{asystA5,4}\\
			\alpha\beta\gamma&=&0\nonumber. 
		\end{eqnarray}.
		
		Since $\beta, \gamma>0$, we obtain from the fourth equation of system \eqref{asystA5,4}, that $\alpha=0$. Therefore, the second and third equations of system \eqref{asystA5,4} become	
		\begin{eqnarray*}
			(3-\lambda_0)\beta^2+(1-\lambda_0)\gamma^2+2c&=&0,\\
			(1-\lambda_0)\beta^2+(3-\lambda_0)\gamma^2+2c&=&0.
		\end{eqnarray*}Hence $2\beta^{2}-2\gamma^{2}=0$. i.e. $ \beta=\gamma. $
	\end{proof} 
	\begin{Corollary}
		An inner product $\langle,\rangle$ on $\mathfrak{g}=A_{5,4}$ is a nilsoliton if and only if $\alpha=0$ and $\beta=\gamma.$ 
	\end{Corollary}
	\begin{Proposition}
		Any inner product on $A_{3,1}\oplus 2A_{1}$ is a Schouten like metric.
	\end{Proposition}

	\begin{proof}
		Let $\langle,\rangle$ be an inner product on $A_{3,1}\oplus 2A_{1}$. In \cite{foka}, the Ricci operator is given by  \begin{equation}
			Ric_{\langle,\rangle}= -\frac{1}{2}diag(\alpha^2,\alpha^2,0,0,-\alpha^2).
		\end{equation}
		We can write $D$ as:	
		\begin{equation}
			\left\{{\renewcommand{\arraystretch}{1}\begin{array}{%
						l @{\qquad} r @{~~} l}
					Dv_1=-\frac{1}{2}[(1-\lambda_0)\alpha^2+2c ]v_1\\
					Dv_2=-\frac{1}{2}[(1-\lambda_0)\alpha^2+2c]v_2 \\
					Dv_3=-\frac{1}{2}(\lambda_0+2c )v_3\\
					Dv_4=-\frac{1}{2}(\lambda_02c) v_4  &\\
					Dv_5=\frac{1}{2}[(1+\lambda_0)\alpha^2-2c] v_5. 
			\end{array}}\right.
		\end{equation}	
		By using \eqref{deriv}, and making tedious calculations, we have:
		\begin{equation}\label{asystA5,1}
			\alpha[(3-\lambda_0)\alpha^2+2c]=0.
		\end{equation}
		Then taking $c=-\frac{\lambda_0-3}{2}\alpha^2$, we have $Ric_{\langle,\rangle}-(\lambda_0s+c) Id\in \mathbb{R}\oplus Der(\mathfrak{g})$.
	\end{proof}
	\begin{Corollary}
		Any inner product on $A_{3,1}\oplus 2A_{1}$ is a nilsoliton. 
	\end{Corollary}
	\begin{Proposition}
		An inner product $\langle,\rangle$ on $\mathfrak{g}=A_{4,1}\oplus A_1$ is a Schouten like metri if and only if  $\gamma=0$ and $\alpha=\beta.$ 
	\end{Proposition}
	\begin{proof}
		Let $\langle,\rangle$ be an inner product on $A_{4,1}\oplus A_1$. From reference \cite{foka}, 
		\begin{equation*}\tiny
			Ric_{\langle,\rangle}=-\frac{1}{2}{\renewcommand{\arraystretch}{0.3}\begin{pmatrix}
					\alpha^2+\beta^2+\gamma^2&0&0&0&0\\
					0&\alpha^2+\gamma^2&\beta\gamma&0&0\\
					0&\beta\gamma&\beta^2-\alpha^2&0&-\alpha\gamma\\
					0&0&0&0&0\\
					0&0&-\alpha\gamma&0&-\beta^2-\gamma^2
			\end{pmatrix}}\; \text{or} \; Ric_{\langle,\rangle}=-\frac{1}{2}{\renewcommand{\arraystretch}{0.3}\begin{pmatrix}
					\alpha^2+\beta^2+\gamma^2&0&0&0&0\\
					0&\alpha^2+\gamma^2&0&0&0\\
					0&0&\alpha^2-\beta^2&-\alpha\gamma&0\\
					0&0&-\alpha\gamma&-\gamma^2&0\\
					0&0&0&0&-\beta^2
			\end{pmatrix}}.
		\end{equation*}
		\begin{itemize}
			\item [i)]If \begin{equation*}\small
				Ric_{\langle,\rangle}=-\frac{1}{2}{\renewcommand{\arraystretch}{0.3}\begin{pmatrix}
						\alpha^2+\beta^2+\gamma^2&0&0&0&0\\
						0&\alpha^2+\gamma^2&\beta\gamma&0&0\\
						0&\beta\gamma&\beta^2-\alpha^2&0&-\alpha\gamma\\
						0&0&0&0&0\\
						0&0&-\alpha\gamma&0&-\beta^2-\gamma^2
				\end{pmatrix}}.
			\end{equation*}
			
			A simple calculation shows that:
			\begin{equation*}
				\left\{{\renewcommand{\arraystretch}{1}\begin{array}{%
							l @{\qquad} r @{~~} l}
						Dv_1=-\frac{1}{2}[(1-\lambda_0)\alpha^2+(1-\lambda_0)\beta^2+(1-\lambda_0)\gamma^2+2c] v_1 \\
						Dv_2=-\frac{1}{2}[(1-\lambda_0)\alpha^2-\lambda_0\beta^{2}+(1-\lambda_0)\gamma^2+2c]v_2-\frac{1}{2}\beta\gamma v_3  \\
						Dv_3=-\frac{1}{2}\beta\gamma v_2-\frac{1}{2}[-(1+\lambda_0)\alpha^2+(1-\lambda_0)\beta^{2}-\lambda_0\gamma^2+2c]v_3 +\frac{1}{2}\alpha\gamma v_5 \\
						Dv_4=-\frac{1}{2}[\lambda_0\alpha^{2}+\lambda_0\beta^{2}+\lambda_0\gamma^{2}-2c] v_4  \\
						Dv_5=\frac{1}{2}\alpha\gamma v_3+\frac{1}{2}[\lambda_0\alpha^{2}+(1+\lambda_0)\beta^2+(1+\lambda_0)\gamma^2-2c ]v_5. 
				\end{array}}\right.
			\end{equation*}	
			Hence, \eqref{deriv} yields
			, $\langle,\rangle$ is a Schouten like metric if and only if the following system of equations is satisfied
			
			\begin{eqnarray}
				\alpha[(3-\lambda_0)\alpha^2-\lambda_0\beta^{2}+(3-\lambda_0)\gamma^2+2c ]&=&0\nonumber\\
				\beta[-\lambda_0\alpha^{2}+(3-\lambda_0)\beta^2+(3-\lambda_0)\gamma^2+2 ]=0\nonumber \\
				\gamma[(2-\lambda_0)\alpha^2+(3-\lambda_0)\beta^2+(3-\lambda_0)\gamma^2+2c]&=&0\label{asystA4,1}\\
				\alpha\beta\gamma&=&0\nonumber. 
			\end{eqnarray}
			Since $\alpha, \beta>0$, then from the fourth equation of system \eqref{asystA4,1}, $\gamma=0$. Therefore, the system become	\begin{eqnarray*}
				(3-\lambda_0)\alpha^2-\lambda_0\beta^2+2c&=&0,\\
				-\lambda_0\alpha^2+(3-\lambda_0)\beta^2+2c&=&0.
			\end{eqnarray*}
			Then, $\alpha=\beta$.
			
			\item [ii)]If \begin{equation*}\small
				Ric_{\langle,\rangle}=-\frac{1}{2}{\renewcommand{\arraystretch}{0.3}\begin{pmatrix}
						\alpha^2+\beta^2+\gamma^2&0&0&0&0\\
						0&\alpha^2+\gamma^2&0&0&0\\
						0&0&\alpha^2-\beta^2&-\alpha\gamma&0\\
						0&0&-\alpha\gamma&-\gamma^2&0\\
						0&0&0&0&-\beta^2
				\end{pmatrix}},
			\end{equation*}
			
			we can then write $D$ as
			\begin{equation*}
				\left\{{\renewcommand{\arraystretch}{1}\begin{array}{%
							l @{\qquad} r @{~~} l}
						Dv_1=-\frac{1}{2}[(1-3\lambda_0)\alpha^2+(1+\lambda_0)\beta^2+(1-\lambda_0)\gamma^2+2c] v_1 \\
						Dv_2=-\frac{1}{2}[(1-3\lambda_0)\alpha^2+\lambda_0\beta^{2}+(1-\lambda_0)\gamma^2+2c]v_2 \\
						Dv_3=-\frac{1}{2}[(1-3\lambda_0)\alpha^2+(\lambda_0-1)\beta^{2}-\lambda_0\gamma^2+2c]v_3 +\frac{1}{2}\alpha\gamma v_4 \\
						Dv_4=\frac{1}{2}\alpha\gamma v_3+\frac{1}{2}[3\lambda_0\alpha^{2}-\lambda_0\beta^{2}+(1+\lambda_0)\gamma^{2}-2c] v_4  \\
						Dv_5=+\frac{1}{2}[3\lambda_0\alpha^{2}+(1-\lambda_0)\beta^2+\lambda_0\gamma^2-2c ]v_5. 
				\end{array}}\right.
			\end{equation*}
			Thus, \eqref{deriv} now acquiesces 
			, $\langle,\rangle$ is a Schouten like metric if and only if the subsequent system is satisfied
			\begin{eqnarray}
				\alpha[(1-3\lambda_0)\alpha^2+(2+\lambda_0)\beta^{2}+(3-\lambda_0)\gamma^2+2c ]&=&0\nonumber\\
				\beta[(2-3\lambda_0)\alpha^{2}+(1+\lambda_0)\beta^2+(1+3\lambda_0)\gamma^2+2 ]&=&0 \nonumber\\
				\gamma[3(1-\lambda_0)\alpha^2+(1+\lambda_0)\beta^2+(3-\lambda_0)\gamma^2+2c]&=&0\label{asystA4,12}\\
				\alpha\beta\gamma&=&0\nonumber. 
			\end{eqnarray}
			From the fourth equation of system \eqref{asystA4,12}, one have $\gamma=0$. Consequently, the system become	\begin{eqnarray*}
				(1-3\lambda_0)\alpha^2+(2+\lambda_0)\beta^2+2c&=&0,\\
				(2-3\lambda_0)\alpha^2+(1+\lambda_0)\beta^2+2c&=&0.
			\end{eqnarray*}
			Therefore, $\alpha=\beta$.
		\end{itemize}	
	\end{proof}
	\begin{Corollary}
		An inner product $\langle,\rangle$ on $\mathfrak{g}=A_{4,1}\oplus A_1$ is a nilsoliton if and only if  $\gamma=0$ and $\alpha=\beta.$ 
	\end{Corollary}

	\begin{Proposition}
		The Lie algebra $A_{5,6}$ does not admit Schouten like metrics.
	\end{Proposition}

	\begin{proof}
		Let $\langle,\rangle$ be an inner product on $A_{5,6}$. In \cite{foka}, the Ricci operator is given by  \begin{equation}
			Ric_{\langle,\rangle}= -\frac{1}{2}{\renewcommand{\arraystretch}{0.4}\begin{pmatrix}
					\alpha^2+\beta^2+\gamma^2+\delta^2+\varepsilon^2&\delta\sigma&0&0&0\\
					\delta\sigma&\alpha^2+\beta^2&\beta\gamma&0&0\\
					0&\beta\gamma&\gamma^2+\delta^2+\sigma^2-\alpha^2&
					\delta\varepsilon-\alpha\beta&0\\ 
					0&0&\delta\varepsilon-\alpha\beta&\varepsilon^2-\beta^2-\gamma^2&-\delta\gamma\\
					0&0&0&-\delta\gamma&-\delta^2-\varepsilon^2-\sigma^2
			\end{pmatrix}}.
		\end{equation}
		Assuming that $\langle,\rangle$ is a Schouten like metric, there exist a real number $c$ and a derivation $D$ on $A_{5,6}$ such that $Ric_{\langle,\rangle}=(\lambda_0s+c )Id+D$. A simple calculation shows that
		\begin{equation*}
			\small	\left\{{\renewcommand{\arraystretch}{1}\begin{array}{%
						l @{\qquad} r @{~~} l}
					Dv_1=-\frac{1}{2}[(1-\lambda_0)\alpha^2+(1-\lambda_0)\beta^2+(1-\lambda_0)\gamma^2+(1-\lambda_0)\delta^2+(1-\lambda_0)\varepsilon^2+2c] v_1-\frac{1}{2}\varepsilon\sigma v_2\\
					Dv_2=-\frac{1}{2}\delta\sigma v_1-\frac{1}{2}[(1-\lambda_0)\alpha^2+(1-\lambda_0)\beta^2-\lambda_0\gamma^2-\lambda_0\delta^2-\lambda_0\varepsilon^2+2c]v_2-\frac{1}{2}\beta\gamma v_3 \\
					
					Dv_3=-\frac{1}{2}\beta\gamma v_2 -\frac{1}{2}[-(1+\lambda_0)\alpha^2-\lambda_0\beta^2+(1-\lambda_0)\gamma^2+(1-\lambda_0)\delta^2-\lambda_0\varepsilon^2+\sigma^2+2c ]v_3-\frac{1}{2}(\delta\varepsilon-\alpha\beta) v_4\\
					
					Dv_4=-\frac{1}{2}(\delta\varepsilon-\alpha\beta) v_3+\frac{1}{2}[\lambda_0\alpha^2+(1+\lambda_0)\beta^2+(1+\lambda_0)\gamma^2+\lambda_0\delta^2+(\lambda_0-1)\varepsilon^2-2c]v_4+\frac{1}{2}\delta\gamma v_5\\
					Dv_5=\frac{1}{2}\delta\gamma v_4 +\frac{1}{2}[\lambda_0\alpha^2+\lambda_0\beta^2+\lambda_0\gamma^2+(1+\lambda_0)\delta^2+(1+\lambda_0)\varepsilon^2+\sigma^2-2c]v_5. 
			\end{array}}\right.
		\end{equation*}		
		Let apply \eqref{deriv}, to the following system:
		\begin{eqnarray}
			[v_1,v_5]&=&0\nonumber \\
			\left[ v_2,v_4\right]  &=&0.\nonumber
		\end{eqnarray}
		We obtaint:
		\begin{eqnarray*}	\delta\gamma\varepsilon&=&0\\
			\sigma(2\delta\varepsilon-\alpha\beta)&=&0
		\end{eqnarray*}
		Therefore $\beta=\delta=0$ inasmuch as $\alpha<0$, $\gamma, \sigma, \varepsilon>0$ . Then, the relation\\ $[v_1,v_3]=\gamma v_4+\delta v_5$ becomes:
		\begin{equation}
			-\frac{1}{2}\gamma[-\lambda_0\alpha^2+(3-\lambda_0)\gamma^2-\lambda_0\varepsilon^2+\sigma^{2}+2c] v_4+\frac{1}{2}\varepsilon\sigma^{2} v_5=0.\label{e4}
		\end{equation}
		Hence, $\varepsilon\sigma^{2}=0$. This contradict the fact that  $\varepsilon, \sigma>0$.
	\end{proof}
	\begin{Corollary}
		The Lie algebra $A_{5,6}$ does not admit nilsoliton.
	\end{Corollary}
	
	\begin{Proposition}
		An inner product $\langle,\rangle$ on $\mathfrak{g}=A_{5,5}$ is a Schouten like metric if and only if $\beta=\delta=0$ and $\alpha=\varepsilon=\sqrt{2}\gamma.$ 
	\end{Proposition}
	\begin{proof}
		Let $\langle,\rangle$ be an inner product on $A_{5,5}$. In \cite{foka}, the Ricci operator is given by  \begin{equation}
			Ric_{\langle,\rangle}= -\frac{1}{2}{\renewcommand{\arraystretch}{0.4}\begin{pmatrix}
					\alpha^2+\beta^2+\gamma^2&\gamma\delta&-\beta\delta&-\beta\varepsilon&0\\
					\gamma\delta&\alpha^2+\beta^2+\delta^2+\varepsilon^2&\beta\gamma&0&0\\
					-\beta\delta&\beta\gamma&\gamma^2+\delta^2&\delta\varepsilon&0\\
					-\beta\varepsilon&0&\delta\varepsilon&\varepsilon^2-\alpha^2&-\alpha\beta\\
					0&0&0&-\alpha\beta&-\beta^2-\gamma^2-\delta^2-\varepsilon^2
			\end{pmatrix}}.
		\end{equation}
		A simple calculation shows that
		\begin{equation*}
			\small	\left\{{\renewcommand{\arraystretch}{1}\begin{array}{%
						l @{\qquad} r @{~~} l}
					Dv_1=-\frac{1}{2}[(1-\lambda_0)\alpha^2+(1-\lambda_0)\beta^2+(1-\lambda_0)\gamma^2-\lambda_0\delta^2-\lambda_0\varepsilon^2+2c] v_1-\frac{1}{2}\gamma\delta v_2+\frac{1}{2}\beta\delta v_3+\frac{1}{2}\beta\varepsilon v_4\\
					Dv_2=-\frac{1}{2}\gamma\delta v_1-\frac{1}{2}[(1-\lambda_0)\alpha^2+(1-\lambda_0)\beta^2-\lambda_0\gamma^2+(1-\lambda_0)\delta^2+(1-\lambda_0)\varepsilon^2+2c]v_2-\frac{1}{2}\beta\gamma v_3 \\
					Dv_3=\frac{1}{2}\beta\delta v_1-\frac{1}{2}\beta\gamma v_2-\frac{1}{2}[-\lambda_0\alpha^2-\lambda_0\beta^2+(1-\lambda_0)\gamma^2+(1-\lambda_0)\delta^2-\lambda_0\varepsilon^2+2c ]v_3+\frac{1}{2}\delta\varepsilon v_4\\
					Dv_4=\frac{1}{2}\beta\varepsilon v_1-\frac{1}{2}\delta\varepsilon v_3-\frac{1}{2}[-(1+\lambda_0)\alpha^2-\lambda_0\beta^2-\lambda_0\gamma^2-\lambda_0\delta^2+(1-\lambda_0)\varepsilon^2+2c]v_4+\frac{1}{2}\alpha\beta v_5\\
					Dv_5=\frac{1}{2}\alpha\beta v_4 +\frac{1}{2}[\lambda_0\alpha^2+(1+\lambda_0)\beta^2+(1+\lambda_0)\gamma^2+(1+\lambda_0)\delta^2+(1+\lambda_0)\varepsilon^2-2c]v_5. 
			\end{array}}\right.
		\end{equation*}	
		In this way, \eqref{deriv} is satisfied if and only if
		\begin{eqnarray}\label{asystA5,5}
			\alpha[(3-\lambda_0)\alpha^2+(3-\lambda_0)\beta^2+(1-\lambda_0)\gamma^2+(1-\lambda_0)\delta^2-\lambda_0\varepsilon^2+2c]&=&0\nonumber\\
			\beta[(3-\lambda_0)\alpha^2+(3-\lambda_0)\beta^2+(3-\lambda_0)\gamma^2+(3-\lambda_0)\delta^2+(3-\lambda_0)\varepsilon^2+2c]&=&0\nonumber \\
			\gamma[(1-\lambda_0)\alpha^2+(3-\lambda_0)\beta^2+(3-\lambda_0)\gamma^2+(3-\lambda_0)\delta^2+(1-\lambda_0)\varepsilon^2+2c]&=&0\nonumber\\
			\delta[(1-\lambda_0)\alpha^2+(3-\lambda_0)\beta^2+(3-\lambda_0)\gamma^2+(3-\lambda_0)\delta^2+(3-\lambda_0)\varepsilon^2+2c]&=&0\nonumber\\ 
			\varepsilon[\lambda_0\alpha^2+(3-\lambda_0)\beta^2+(1-\lambda_0)\gamma^2+(3-\lambda_0)\delta^2+(3-\lambda_0)\varepsilon^2+2c]&=&0\nonumber\\
			\alpha\beta\gamma=0\\
			\alpha\beta\delta=0\nonumber\\
			\alpha\beta\varepsilon=0\nonumber\\
			\alpha\delta\varepsilon=0\nonumber\\
			\gamma\delta\varepsilon=0\nonumber\\
			\beta\gamma\varepsilon=0\nonumber.
		\end{eqnarray}
		Since $\gamma, \varepsilon>0$, we have $\beta=\delta=0$. The previous system is transform into
		Since $\gamma, \varepsilon>0$, we have $\beta=\delta=0$. System \eqref{asystA5,3}
		is equivalent to  \begin{eqnarray}
			(3-\lambda_0)\alpha^2+(1-\lambda)_0\gamma^2-\lambda_0\varepsilon^2+2c&=&0\nonumber\\
			(1-\lambda_0)\alpha^2+(3-\lambda_0)\gamma^2+(1-\lambda_0)\varepsilon^2+2c&=&0\quad\quad\quad\quad\quad\quad\quad\quad\quad\quad\quad\\
			\lambda_0\alpha^2+(1-\lambda_0)\gamma^2+(3-\lambda_0)\varepsilon^2+2c&=&0\nonumber.
		\end{eqnarray}
		
		Therefore, $\alpha=\varepsilon=\sqrt{2}\gamma$.
	\end{proof}
	
	\begin{Corollary}
		An inner product $\langle,\rangle$ on $\mathfrak{g}=A_{5,5}$ is a nilsoliton if and only if $\beta=\delta=0$ and $\alpha=\varepsilon=\sqrt{2}\gamma.$ 
	\end{Corollary}
	
	\begin{Proposition}
		An inner product $\langle,\rangle$ on $\mathfrak{g}=A_{5,3}$ is a Schouten like metric if and only if $\beta=\delta=0$ and $\varepsilon=\gamma=\frac{\sqrt{3}}{2}\alpha.$ 
	\end{Proposition}
	\begin{proof}
		Let $\langle,\rangle$ be an inner product on $A_{5,3}$. In \cite{foka}, the Ricci operator is given by  \begin{equation}
			Ric_{\langle,\rangle}=-\frac{1}{2}{\renewcommand{\arraystretch}{0.4}\begin{pmatrix}
					\alpha^2+\beta^2+\gamma^2+\delta^2&\delta\varepsilon&0&0&0\\
					\delta\varepsilon&\alpha^2+\beta^2+\varepsilon^2&\beta\gamma&0&0\\
					0&\beta\gamma&\gamma^2+\delta^2+\varepsilon^2-\alpha^2&-\alpha\beta&0\\
					0&0&-\alpha\beta&-\beta^2-\gamma^2&-\delta\gamma\\
					0&0&0&-\delta\gamma&-\delta^2-\varepsilon^2
			\end{pmatrix}}.
		\end{equation}
		A simple calculation shows that
		\begin{equation*}
			\small	\left\{{\renewcommand{\arraystretch}{1}\begin{array}{%
						l @{\qquad} r @{~~} l}
					Dv_1=-\frac{1}{2}[(1-\lambda_0)\alpha^2+(1-\lambda_0)\beta^2+(1-\lambda_0)\gamma^2+(1-\lambda_0)\delta^2-\lambda_0\varepsilon^2+2c] v_1-\frac{1}{2}\delta\varepsilon v_2\\
					Dv_2=-\frac{1}{2}\delta\varepsilon v_1-\frac{1}{2}[(1-\lambda_0)\alpha^2+(1-\lambda_0)\beta^2-\lambda_0\gamma^2-\lambda_0\delta^2+(1-\lambda_0)\varepsilon^2+2c]v_2-\frac{1}{2}\beta\gamma v_3 \\
					Dv_3=-\frac{1}{2}\beta\gamma v_2-\frac{1}{2}[-(1+\lambda_0)\alpha^2-\lambda_0\beta^2+(1-\lambda_0)\gamma^2+(1-\lambda_0)\delta^2+(1-\lambda_0)\varepsilon^2+2c ]v_3+\frac{1}{2}\alpha\beta v_4\\
					Dv_4=\frac{1}{2}\alpha\beta v_3+\frac{1}{2}[\lambda_0\alpha^2+(1+\lambda_0)\beta^2+(1+\lambda_0)\gamma^2+\lambda_0\delta^2+\lambda_0\varepsilon^2-2c]v_4+\frac{1}{2}\delta\gamma v_5\\
					Dv_5=\frac{1}{2}\delta\gamma v_4 +\frac{1}{2}[\lambda_0\alpha^2+\lambda_0\beta^2+\lambda_0\gamma^2+(1+\lambda_0)\delta^2+(1+\lambda_0)\varepsilon^2-2c]v_5. 
			\end{array}}\right.
		\end{equation*}	
		In this way, \eqref{deriv} is satisfied if and only if
		\begin{eqnarray}\label{asystA5,3}
			\alpha[(3-\lambda_0)\alpha^2+(3-\lambda_0)\beta^2-\lambda_0\gamma^2-\lambda_0\delta^2-\lambda_0\varepsilon^2+2c]&=&0\nonumber\\
			\beta[(3-\lambda_0)\alpha^2+(3-\lambda_0)\beta^2+(3-\lambda_0)\gamma^2+(1-\lambda_0)\delta^2+(1-\lambda_0)\varepsilon^2+2c]&=&0\nonumber \\
			\gamma[-\lambda_0\alpha^2+(3-\lambda_0)\beta^2+(3-\lambda_0)\gamma^2+(3-\lambda_0)\delta^2+(1-\lambda_0)\varepsilon^2+2c]&=&0\nonumber\\
			\delta[-\lambda_0\alpha^2+(1-\lambda_0)\beta^2+(3-\lambda_0)\gamma^2+(3-\lambda_0)\delta^2+(3-\lambda_0)\varepsilon^2+2c]&=&0\nonumber\\ 
			\varepsilon[-\lambda_0\alpha^2+(1-\lambda_0)\beta^2+(1-\lambda_0)\gamma^2+(3-\lambda_0)\delta^2+(3-\lambda_0)\varepsilon^2+2c]&=&0\nonumber\\
			\alpha\beta\gamma&=&0\quad\quad\quad\quad\\
			\alpha\beta\delta&=&0\nonumber\\
			\alpha\beta\varepsilon&=&0\nonumber\\
			\beta\gamma\delta&=&0\nonumber\\
			\gamma\delta\varepsilon&=&0\nonumber.
		\end{eqnarray}
		Since $\gamma, \varepsilon>0$, we have $\beta=\delta=0$. System \eqref{asystA5,3}
		is equivalent to  \begin{eqnarray}
			(3-\lambda_0)\alpha^2-\lambda_0\gamma^2-\lambda_0\varepsilon^2+2c&=&0\nonumber\\
			-\lambda_0\alpha^2+(3-\lambda_0)\gamma^2+(1-\lambda_0)\varepsilon^2+2c&=&0\quad\quad\quad\quad\quad\quad\quad\quad\quad\quad\quad\\
			-\lambda_0\alpha^2+(1-\lambda_0)\gamma^2+(3-\lambda_0)\varepsilon^2+2c&=&0\nonumber.
		\end{eqnarray}
		
		consequently, $\gamma=\varepsilon=\frac{\sqrt{3}}{2}\alpha$.
	\end{proof}
	
	\begin{Corollary}
		An inner product $\langle,\rangle$ on $\mathfrak{g}=A_{5,3}$ is a nilsoliton if and only if $\beta=\delta=0$ and $\varepsilon=\gamma=\frac{\sqrt{3}}{2}\alpha.$ 
	\end{Corollary}
	\begin{Proposition}
		An inner product $\langle,\rangle$ on $\mathfrak{g}=A_{5,1}$ is an algebraic Schouten soliton if and only if $\beta=0$ and $\alpha=\gamma.$ 
	\end{Proposition}
	\begin{proof}
		Let $\langle,\rangle$ be an inner product on $A_{5,1}$.	By 	\cite{foka}, \begin{equation}
			Ric_{\langle,\rangle}=-\frac{1}{2}{\renewcommand{\arraystretch}{0.4}\begin{pmatrix}
					\alpha^2+\beta^2+\gamma^2&0&0&0&0\\
					0&\alpha^2+\beta^2&\beta\gamma&0&0\\
					0&\beta\gamma&\gamma^2&0&0\\
					0&0&0&-\alpha^2&-\alpha\beta\\
					0&0&0&-\alpha\beta&-\beta^2-\gamma^2
			\end{pmatrix}}.
		\end{equation}
		Then, 
		\begin{equation*}
			\left\{{\renewcommand{\arraystretch}{1}\begin{array}{%
						l @{\qquad} r @{~~} l}
					Dv_1=-\frac{1}{2}[(1-\lambda_0)\alpha^2+(1-\lambda_0)\beta^2+(1-\lambda_0)\gamma^2+2c] v_1\\
					Dv_2=-\frac{1}{2}[(1-\lambda_0)\alpha^2+(1-\lambda_0)\beta^2+\lambda_0\gamma^{2}+2c]v_2 -\frac{1}{2}\beta\gamma v_3& \\
					Dv_3=-\frac{1}{2}\beta\gamma v_2-\frac{1}{2}[-\lambda_0\alpha^{2}-\lambda_0\beta^{2}+(1-\lambda_0)\gamma^2+2c] v_3 &\\
					Dv_4=\frac{1}{2}[(1+\lambda_0)\alpha^{2}+\lambda_0\beta^{2}+\lambda_0\gamma^{2}-2c] v_4+\frac{1}{2}\alpha\beta v_5  &\\
					Dv_5=\frac{1}{2}\alpha\beta v_4+\frac{1}{2}[\lambda_0\alpha^{2}+(1+\lambda_0)\beta^2+(1+\lambda_0)\gamma^2-2c ]v_5. 
			\end{array}}\right.
		\end{equation*}	
		Ergo, Equation \eqref{deriv} is satisfied if and only if
		\label{asystA5,1}
		\begin{eqnarray}
			\alpha[(3-\lambda_0)\alpha^2+(3-\lambda_0)\beta^{2}+(1-\lambda_0)\gamma^2+2c ]&=&0\nonumber\\
			\beta[(3-\lambda_0)\alpha^{2}+(3-\lambda_0)\beta^2+(3-\lambda_0)\gamma^2+2 ]&=&0 \nonumber\\
			\gamma[(1-\lambda_0)\alpha^2+(3-\lambda_0)\beta^2+(3-\lambda_0)\gamma^2+2c]&=&0\label{asystA5,1}\\
			\alpha\beta\gamma&=&0\nonumber.
		\end{eqnarray}
		
		By the fourth equation in the previous system, we get $\beta=0$. Hence, system  \eqref{asystA5,1} becomes \begin{eqnarray*}
			(3-\lambda_0)\alpha^2+(1-\lambda_0)\gamma^2+2c&=&0,\\
			(1-\lambda_0)\alpha^2+(3-\lambda_0)\gamma^2+2c&=&0.
		\end{eqnarray*} And then, one obtains $\alpha=\gamma$.
	\end{proof}
	
	\begin{Corollary}
		An inner product $\langle,\rangle$ on $\mathfrak{g}=A_{5,1}$ is a nilsoliton if and only if $\beta=0$ and $\alpha=\gamma.$
	\end{Corollary}
	\begin{Proposition}
		An inner product $\langle,\rangle$ on $\mathfrak{g}=A_{5,2}$ is a Schouten like metric if and only if $\beta=0$ and $\alpha=\delta=\frac{\sqrt{3}}{2}\gamma.$ 
	\end{Proposition}
	\begin{proof}
		Let $\langle,\rangle$ be an inner product on $A_{5,2}$.	According to \cite{foka}, \begin{equation}
			Ric_{\langle,\rangle}=-\frac{1}{2}{\renewcommand{\arraystretch}{0.4}\begin{pmatrix}
					\alpha^2+\beta^2+\gamma^2+\delta^2&0&0&0&0\\
					0&\alpha^2+\beta^2&\beta\gamma&0&0\\
					0&\beta\gamma&\gamma^2-\alpha^2&-\alpha\beta&0\\
					0&0&-\alpha\beta&\delta^2-\beta^2-\gamma^2&0\\
					0&0&0&0&-\delta^2
			\end{pmatrix}}.
		\end{equation}
		We have
		\begin{equation*}
			\left\{{\renewcommand{\arraystretch}{1}\begin{array}{%
						l @{\qquad} r @{~~} l}
					Dv_1=-\frac{1}{2}[(1-\lambda_0)\alpha^2+(1-\lambda_0)\beta^2+(1-\lambda_0)\gamma^2+(1-\lambda_0)\delta^2+2c] v_1\\
					Dv_2=-\frac{1}{2}[(1-\lambda_0)\alpha^2+(1-\lambda_0)\beta^2-\lambda_0\gamma^{2}-\lambda_0\delta^{2}+2c]v_2 -\frac{1}{2}\beta\gamma v_3& \\
					Dv_3=-\frac{1}{2}\beta\gamma v_2-\frac{1}{2}[-(1+\lambda_0)\alpha^{2}-\lambda_0\beta^{2}+(1-\lambda_0)\gamma^2-\lambda_0\delta+2c] v_3+\frac{1}{2}\alpha\beta v_4 &\\
					Dv_4=\frac{1}{2}\alpha\beta v_3-\frac{1}{2}[-\lambda_0\alpha^{2}-(1+\lambda_0)\beta^{2}-(1+\lambda_0)\gamma^{2}+(1-\lambda_0)\delta^{2}+2c] v_4  &\\
					Dv_5=\frac{1}{2}[\lambda_0\alpha^{2}+\lambda_0\beta^2+\lambda_0\gamma^2+(1+\lambda_0)\delta^{2}-2c ]v_5. 
			\end{array}}\right.
		\end{equation*}	
		Thus, Equation \eqref{deriv} is satisfied if and only if
		\label{asystA5,2}
		\begin{eqnarray}
			\delta[(1-\lambda_0)\alpha^2-\lambda_0\beta^{2}-\lambda_0\gamma^{2}+(3-\lambda_0)\delta^2+2c ]&=&0\nonumber\\
			\beta[(3-\lambda_0)\alpha^2+(3-\lambda_0)\beta^2+(3-\lambda_0)\gamma^2-\lambda_0\delta^{2}+2c] &=&0\nonumber\\
			\alpha[(3-\lambda_0)\alpha^2+(3-\lambda_0)\beta^2+(1-\lambda_0)\gamma^{2}-\lambda_0\delta^2+2c] &=&0 \label{asystA5,2}\\
			\gamma[-\lambda_0\alpha^{2}+(3-\lambda_0)\beta^2+(3-\lambda_0)\gamma^2-\lambda_0\delta^{2}+2c]&=&0\nonumber\\
			\alpha\beta\gamma&=&0\nonumber\\ 
			\alpha\beta\delta&=&0\nonumber.
		\end{eqnarray}.
		
		Due to the fact that $\alpha, \gamma>0$, the fifth equation of system \eqref{asystA5,2} is equivalent to $\beta=0$. Accordingly  \begin{eqnarray*}
			(3-\lambda_0)\alpha^2-\lambda_0\gamma^2+(1-\lambda_0)\delta^2+2c&=&0,\\
			-\lambda_0\alpha^2+(3-\lambda_0)\gamma^2-\lambda_0\delta^2+2c&=&0,\\
			(1-\lambda_0)\alpha^2-\lambda_0\gamma^2+(3-\lambda_0)\delta^2+2c&=&0.
		\end{eqnarray*}
		Then \begin{eqnarray*}
			3\alpha^2-3\gamma^2+\delta^2&=&0,\\
			\alpha^2-3\gamma^2+3\delta^2&=&0,\\
			2\alpha^2-2\delta^2+2c&=&0.
		\end{eqnarray*}Therefore, the first and second equations among the previous three equations give $\alpha=\delta$. Then the last one gives $\alpha=\frac{\sqrt{3}}{2}\gamma.$
	\end{proof}
	\begin{Corollary}
		An inner product $\langle,\rangle$ on $\mathfrak{g}=A_{5,2}$ is a nilsoliton if and only if $\beta=0$ and $\alpha=\delta=\frac{\sqrt{3}}{2}\gamma.$ 
	\end{Corollary}


\begin{thebibliography}{99}
		\bibitem{buc} M.Boucetta.Curvature of left invariant riemannian metrics on lie groups.advance in mathematics, 1976. \textit{Manuscripta Math}. \textbf{135}(1-2) (2011) 229-243.
		\bibitem{butt} T. Buttsworth. The prescibe Ricci curvature problem on three-dimensional Lie groups. Mathematische Nachrichten. 2018;1-13. 
		\bibitem{bp}T. Buttsworth, A. Pulemotov. The Prescribed Ricci Curvature Problem for Homogeneous Metrics. In: Dearricott O, Tuschmann W, Nikolayevsky Y, Leistner T, Crowley D, eds. Differential Geometry in the Large. London Mathematical Society Lecture Note Series. Cambridge University Press; 2020:169-192.
		\bibitem{cal}Calvino-Louzao, E.; Hervella, L.M.; Seoane-Bascoy, J.; Vazquez-Lorenzo, R. Homogeneous Cotton solitons. J. Phys. A Math. Theor.
		2013, 46, 285204
	\bibitem{foka} Marius Landry FOKA et al.,The prescribed Ricci curvature problem on 5-dimensional nilpotent Lie groups, International Journal of Geometric Methods in Modern Physics, https://doi.org/10.1142/S0219887825500550.
		
		\bibitem{si} Liu, Siyao. "Algebraic Schouten Solitons of Three-Dimensional Lorentzian Lie Groups." Symmetry 15.4 (2023): 866. 
		\bibitem{hash} T. Hashinaga and H. Tamaru, Three-dimensional solvsolitons and the minimality of the corresponding submanifolds, preprint, arXiv:1501.05513.
		\bibitem{taka}  T. Hashinaga, H. Tamaru and K. Terada, Milnor-type theorems for left-invariant Riemannian metrics on Lie groups, \textit{J. Math. Soc. Japan} \textbf{68}(2) (2016) 669-684.
		
		\bibitem{mil} J.Milnor. Curvature of left invariant metrics on lie groups.advance in mathematics, 1976.
		\bibitem{mo}	MOGHADDAM, Hamid Reza Salimi. On the classification of five-dimensional nilsolitons. arXiv preprint arXiv:1912.13322, 2019.
		\bibitem{pu1} A. Pulemotov. Metrics with prescribed Ricci curvature near the boundary of a manifold. Math.
		
		\bibitem{koda}  A. Takahara and H. Tamaru, The space of left-invariant metrics on a Lie group up to isometry and scaling, Manuscripta Math. 135(1-2) (2011) 229-243.
		\bibitem{we}WEARS, Thomas H. On algebraic solitons for geometric evolution equations on three-dimensional Lie groups. 2016.	
		
	\end{thebibliography}
\end{document}